\documentclass[a4paper,11pt]{article}
\usepackage{amsmath, amsfonts, amscd, amssymb, amsthm, enumerate}

\DeclareMathOperator{\id}{\operatorname{id}}

\DeclareMathOperator{\Mat}{\operatorname{M}}
\DeclareMathOperator{\MatH}{\operatorname{\mathcal{H}}}
\DeclareMathOperator{\Matah}{\operatorname{A\mathcal{H}}}

\DeclareMathOperator{\Mata}{\operatorname{A}}
\DeclareMathOperator{\Mats}{\operatorname{S}}

\DeclareMathOperator{\GL}{\operatorname{GL}}

\DeclareMathOperator{\car}{\operatorname{char}}

\renewcommand{\setminus}{\smallsetminus}
\renewcommand{\epsilon}{\varepsilon}

\def\F{\mathbb{F}}

\def\D{\mathbb{D}}

\def\calF{\mathcal{F}}

\def\calM{\mathcal{M}}
\def\calN{\mathcal{N}}

\def\calZ{\mathcal{Z}}

\def\lcro{\mathopen{[\![}}
\def\rcro{\mathclose{]\!]}}

\theoremstyle{definition}

\theoremstyle{plain}
\newtheorem{theo}{Theorem}[section]

\newtheorem{lemma}[theo]{Lemma}

\theoremstyle{plain}

\theoremstyle{remark}

\title{Range-compatible homomorphisms on Hermitian matrices}
\author{Cl\'ement de Seguins Pazzis\footnote{Universit\'e de Versailles Saint-Quentin-en-Yvelines, Laboratoire de Math\'ematiques
de Versailles, 45 avenue des Etats-Unis, 78035 Versailles cedex, France}
\footnote{e-mail address: clement.de-seguins-pazzis@ac-versailles.fr}}

\begin{document}

\thispagestyle{plain}

\maketitle
\begin{abstract}
Let $\D$ be a division ring with an involution $x \mapsto x^\star$, and $n \geq 2$ be an integer.
Denote by $\MatH_n(\D)$ the set of all $n$-by-$n$ Hermitian matrices with entries in $\D$, and by 
$\Matah_n(\D)$ the set of all matrices $A-A^\star$ with $A \in \Mat_n(\D)$.
Here, we give a complete solution to the following problem:
Determine all group homomorphisms from $\MatH_n(\D)$ to $\D^n$ (respectively, from $\Matah_n(\D)$ to $\D^n$ unless $(-)^\star$ is the identity) that take every matrix to a right linear combination of its columns.

The solution to this problem was already known when $(-)^\star$ is the identity, and the novelty here lies in the generalization to arbitrary involutions, and in particular
in the noncommutative case.

These results are to be used in a subsequent article on subspaces of Hermitian matrices of bounded rank, and on large spaces of diagonalisable matrices.
\end{abstract}




\section{Introduction}

\subsection{Introduction to range-compatible homomorphisms}

This article deals with range-compatible homomorphisms on matrix spaces, a concept that first emerged in
\cite{dSPclassIsrael} as a tool in the study of large linear subspaces of matrices, and was first investigated in depth in \cite{dSPRC1}. 
Since then, its study has been greatly expanded \cite{dSPRCsym,dSPFlandersskew,dSPRCaffine,dSPRC2}, and numerous applications to spaces
of bounded rank matrices have been obtained \cite{dSPclassIsrael,dSPaffinesym,dSPlargerankrevisited}.

Throughout, we let $\D$ be a division ring, that is an associative algebra with a unity $1_\D \neq 0_\D$ 
in which every nonzero element is invertible. Right-$\D$-modules (respectively, left-$\D$-modules) are simply called
right-$\D$-vector spaces (respectively, left-$\D$-vector spaces), and we naturally see $\D^n$, the set of all $n$-lists of elements of $\D$,
which we identify with column matrices with $n$ rows, as a right $\D$-vector space through the scalar multiplication
$$\begin{bmatrix}
x_1 & \cdots & x_n
\end{bmatrix}^T\,a:=
\begin{bmatrix}
x_1 a & \cdots & x_n a 
\end{bmatrix}^T.$$
 We define the column space, or range, of a matrix
$M \in \Mat_{n,p}(\D)$ (with $n$ rows, $p$ columns and entries in $\D$) as the linear span of its columns in the right $\D$-vector space $\D^n$.
We will denote by $(E_1,\dots,E_n)$ the standard basis of $\D^n$.

Given a subset $\calM$ of $\Mat_{n,p}(\D)$, a \textbf{range-compatible mapping} on $\calM$ is 
a mapping $F : \calM \rightarrow \D^n$ that assigns to every matrix of $\calM$ a vector of its range.
Now, assume that $\calM$ is a subgroup of $\Mat_{n,p}(\D)$. A \textbf{range-compatible homomorphism} on $\calM$
is a range-compatible mapping that is also a group homomorphism from $(\calM,+)$ to $(\D^n,+)$.
The set of all such mappings is then not only a subgroup of the group $(\calF(\calM,\D^n),+)$ of all mappings from $\calM$ to $\D^n$, but also a linear subspace of the 
right $\D$-vector space $\calF(\calM,\D^n)$ (simply, if $F$ is a range-compatible homomorphism then so is $F \alpha$ for every $\alpha \in \D$).

A basic example of range-compatible homomorphisms are the so-called local mappings. A 
mapping $F : \calM \rightarrow \D^n$ is called \textbf{local} whenever there exists $X \in \D^p$
such that $\forall M \in \calM, \; F(M)=MX$. Obviously, all local mappings are range-compatible homomorphisms,
and the basic question is whether, for a given $\calM$, the converse holds.

Obviously, for $n=1$ every group homomorphism from $\calM$ to $\D$ is range-compatible,
and more generally this holds true whenever $n=p$ and the nonzero matrices in $\calM$ are invertible: in those cases it is generally possible
to construct range-compatible homomorphisms that are non-local, and it is rather hopeless to find a nice closed form for all
the group homomorphisms from $\calM$ to $\mathbb{D}$.

Range-compatible homomorphisms have been largely studied for the following kinds of spaces:
\begin{itemize}
\item Spaces of rectangular matrices over a field, with small codimension in the full space \cite{dSPclassIsrael,dSPRC1,dSPRC2}.
\item Full spaces of matrices over a division ring: see section 2.2 of \cite{dSPFlandersskew}, where it is assumed that $\D$ is finite-dimensional over its center, but this assumption is not needed at all. For example, if $n \geq 2$ then every range-compatible homomorphism on $\Mat_{n,p}(\D)$ is local. This can be viewed as a simple application of the characterization of the group endomorphisms of $\D^n$ that map every vector to a right scalar multiple of itself.
\item Full spaces of symmetric matrices over a field \cite{dSPRC1}, and more generally spaces of symmetric matrices over a field with small codimension in the corresponding space of symmetric matrices \cite{dSPRCsym}.
\item Full spaces of alternating matrices over a field, and more generally spaces of alternating matrices over a field with small codimension in the corresponding space of alternating matrices \cite{dSPRCsym}.
\end{itemize}

\subsection{Subgroups of Hermitian matrices}

Here we shall consider a form of generalization of the results on symmetric/alternating matrices, so as to obtain the Hermitian case.
Let $R$ be a ring (associative, with a unity $1_R \neq 0_R$), and 
$x \mapsto x^\star$ be an involution of $R$, i.e., an endomorphism of the group $(R,+)$ that satisfies
$$\forall (x,y)\in R^2, \; (xy)^\star=y^\star x^\star \quad \text{and} \quad \forall x \in R, \; (x^\star)^\star=x,$$
which in particular requires that $(1_R)^\star=1_R$.
The set of all \textbf{Hermitian} elements 
$$\{x \in R : x^\star=x\}$$ 
is a subgroup of $(R,+)$, 
and so is the one of all \textbf{alternate-Hermitian} elements 
$$\{x-x^\star \mid x \in R\}.$$

The involution $(-)^\star$ leaves the center $\calZ$ of $R$ invariant, and is called of the first kind if it
is the identity on $\calZ$, and of the second kind otherwise.
The unit group $R^\times$ acts on $R$ by congruence through $x.a:=xax^\star$, and we note that 
the set of all Hermitian elements is invariant under congruence, as well as the set of all alternate-Hermitian elements.

Now, and throughout the article, we endow $\D$ with such an involution $(-)^\star$, which we fix once and for all.
We set
$$H:=\{x \in \D : x^\star=x\} \quad \text{and} \quad A:=\{x^\star-x \mid x \in \D\}.$$
If $\car(\D) \neq 2$ then $A$ also equals the set of all skew-Hermitian elements (the elements $x$ for which $x^\star=-x$),
and we have $H \oplus A=\D$; if $\car(\D)=2$ the skew-Hermitian elements of $\D$ are the Hermitian ones, and they differ from the alternate-Hermitian ones in general.
A notable exception is when $(-)^\star$ is of the second kind and $\car(\D)=2$, in which case we can pick a central element $a_0$ such that $a_0^\star \neq a_0$, and use it to write every Hermitian element $x$ as $x=(\lambda a_0 x)-(\lambda a_0 x)^\star$ for the central Hermitian element $\lambda:=(a_0-a_0^\star)^{-1}$.

The involution $(-)^\star$ induces, for all integers $n>0$ and $p>0$, a bijection $M=(m_{i,j}) \in \Mat_{n,p}(\D) \mapsto M^\star:=(m_{j,i}^\star) \in \Mat_{p,n}(\D)$,
whose reciprocal mapping is obviously $N \in \Mat_{p,n}(\D) \mapsto M^\star \in \Mat_{n,p}(\D)$.
Moreover, we have the obvious identity $(MN)^\star=N^\star M^\star$ for matrices $M,N$ with entries in $\D$ whenever the multiplication makes sense.
In particular, for $n \geq 1$ we recover an involution $M \mapsto M^\star$ of the ring $\Mat_n(\D)$, and we denote:
\begin{itemize}
\item by $\MatH_n(\D)$ the set of all its Hermitian elements;
\item by $\Matah_n(\D)$ the set of all its alternate-Hermitian elements.
\end{itemize}
For example, if $\D$ is a field and $(-)^\star$ is the identity then we simply have 
$\MatH_n(\D)=\Mats_n(\D)$ the set of all symmetric $n$-by-$n$ matrices with entries in $\D$, 
and $\Matah_n(\D)=\Mata_n(\D)$ the set of all alternating $n$-by-$n$ matrices with entries in $\D$
(i.e., skew-symmetric matrices with all diagonal entries equal to $0$).

In general, it is obvious that $\MatH_n(\D)$ consists of all the matrices $H=(h_{i,j})$ such that 
$h_{j,i}^\star=h_{i,j}$ for all distinct $i$ and $j$, and the diagonal entries are Hermitian elements of $\D$,
whereas $\Matah_n(\D)$ consists of all the matrices $A=(a_{i,j})$ such that $a_{j,i}=-a_{i,j}^\star$ for all distinct $i$ and $j$, and the diagonal entries are alternate-Hermitian elements of $\D$ (and not simply skew-Hermitian).
Finally, it is classical that unless $(-)^\star$ is the identity every matrix of $\MatH_n(\D)$ or of $\Matah_n(\D)$
is congruent to a diagonal one. This fails if $(-)^\star$ is the identity, as no nonzero alternating matrix is congruent to a diagonal one \cite{Knus},
yet every non-alternating symmetric matrix with entries in a field is congruent to a diagonal one (see theorem 3.0.13 of \cite[Chapter XXXV]{invitquad}).

We can now state our aim here: classify all range-compatible homomorphisms on $\MatH_n(\D)$ and $\Matah_n(\D)$.
As said earlier, this has been already achieved in the special case where $(-)^\star$ is the identity
(in which case $\D$ is commutative). See \cite{dSPRC1} for the case of symmetric matrices, and \cite{dSPRCsym} for the case of alternating matrices. Here, we will systematically avoid the situation of alternating matrices over a field, and we simply state the result for reference:

\begin{theo}[Special case of theorem 1.7 in \cite{dSPRCsym}]\label{theo:alternate}
Let $\F$ be a field, and $n \geq 3$ be an integer.
Then every range-compatible homomorphism on $\Mata_n(\F)$ is local.
\end{theo}

In case $n=2$, every group homomorphism from $\Mata_2(\F)$ to $\F^2$ is range-compatible since every nonzero alternating $2$-by-$2$
matrix is invertible, and if $\F$ is not a prime field it is possible to construct a group homomorphism that is not $\F$-linear,
whereas every local map is $\F$-linear. Hence the need to single out the case $n=2$.

\subsection{Main results}\label{section:mainresults}

We will now state our main new results.
As we shall see, the difficulty is essentially concentrated in the characteristic $2$ case, as was already the case in the study of range-compatible homomorphisms
on spaces of symmetric matrices. To start with, we deal with the other cases:

\begin{theo}\label{theo:mainHermitian}
Let $\D$ be a division ring with an involution $(-)^\star$.
Let $n \geq 2$ be an integer. 
In each one of the following cases,  every range-compatible homomorphism on $\MatH_n(\D)$ is local:
\begin{enumerate}[(i)]
\item $\car(\D) \neq 2$ ;
\item $(-)^\star$ is of the second kind, unless $n=2$ and $|\D|=4$.
\end{enumerate}
\end{theo}

The special case where $n=2$, $|\D|=4$ and $(-)^\star$ is of the second kind will be discussed later.

\begin{theo}\label{theo:mainaltHermitian}
Let $\D$ be a division ring with an involution $(-)^\star$ that is not the identity.
Let $n \geq 2$ be an integer. 
Then, unless $n=2$ and $|\D|=4$, every range-compatible homomorphism on $\Matah_n(\D)$ is local.
\end{theo}

Now we turn to the difficult case where $\car(\D)=2$, for Hermitian matrices (alternate-Hermitian matrices are entirely handled by Theorems 
\ref{theo:mainaltHermitian} and \ref{theo:alternate} unless $n=2$, $(-)^\star$ is of the second kind and $|\D|=4$).
This involves the following unusual notion:
a mapping $\alpha : H \rightarrow \D$ is called \textbf{root-linear} whenever it is a group homomorphism and
$$\forall (x,a) \in \D \times H, \; \alpha (xax^\star)=x \alpha(a),$$
and we similarly define root-linear mappings from $A$ to $\D$.

We will see shortly (Lemma \ref{lemma:rootlinear}) that every such map must vanish on $A$ provided $\car(\D)=2$, and hence only the zero mapping is root-linear if $(-)^\star$ is of the second kind.
If $\D$ is commutative with characteristic $2$ and $(-)^\star$ is the identity, then the
root-linear mappings are simply the linear forms on the $\D$-vector space with underlying group $(\D,+)$
and scalar multiplication $\lambda\cdot x:=\lambda^2 x$. Hence, nonzero root-linear mappings always exist in that case
(and if $\D$ is perfect those forms constitute a $1$-dimensional $\D$-linear subspace of $\calF(\D,\D)$).

For any square matrix $M=(m_{i,j}) \in \Mat_n(\D)$, we denote by 
$$\Delta(M)=\begin{bmatrix}
m_{1,1} & \cdots & m_{n,n}
\end{bmatrix}^T \in \D^n$$ 
its diagonal vector.

\begin{theo}\label{theo:specialHermitianbis}
Assume that $\car(\D)=2$. Then for every root-linear mapping $\alpha : H \rightarrow \D$,
the mapping $M \in \MatH_n(\D) \mapsto \Delta(M)^\alpha$ is a range-compatible homomorphism, and it is local
only if $\alpha=0$.
\end{theo}

Here of course, for a list $X=(x_1,\dots,x_n)$ in $H^n$ and a mapping $f$ with source set $H$, we write $X^f$ short for $(f(x_1),\dots,f(x_n))$.

\begin{theo}\label{theo:specialHermitian}
Assume that $\car(\D)=2$ and $(-)^\star$ is of the first kind.
Then every range-compatible mapping on $\MatH_n(\D)$ splits as $M \mapsto MX+\Delta(M)^\alpha$ for a unique pair $(X,\alpha)$ consisting of a vector $X \in \D^n$ and of a root-linear mapping $\alpha : H \rightarrow \D$.
\end{theo}

This last result was already known in the commutative case, see section 3 of \cite{dSPRC1}.

We finish with the very special case where $n=2$, $|\D|=4$ and $(-)^\star$ is of the second kind. In that case there is no nonzero root-linear mapping
on $H$ (see Lemma \ref{lemma:rootlinear}) yet 
$$\Phi_0 : \begin{cases}
\MatH_2(\D) & \longrightarrow \D^2 \\
\begin{bmatrix}
a & x^\star \\
x & b
\end{bmatrix} & \longmapsto 
\begin{bmatrix}
x \\
x^\star
\end{bmatrix}
\end{cases}
$$
is a range-compatible homomorphism. That $\Phi_0$ is range-compatible can of course be verified by a case-by-case study
(there are only $16$ relevant matrices), but can also be deduced from the previous type of counterexample.
Let indeed $M=\begin{bmatrix}
a & x^\star \\
x & b
\end{bmatrix} \in \MatH_2(\D)$ be singular and such that $\Phi_0(M) \neq 0$. The latter yields $x \neq 0$, so $xx^\star=1$
and hence $ab=1$ because $M$ is singular, so $a=b=1$.
Then we observe that $\Phi_0(M)=\sqrt{\Delta(N)}$ for $N:=\begin{bmatrix}
x^\star & 1 \\
1 & x
\end{bmatrix}$, which is deduced from $M$ by permuting the two columns and hence has the same column space.
Then as $N$ is a symmetric matrix and $\D$ is commutative with characteristic $2$ we know from Theorem 
\ref{theo:specialHermitianbis} that $\sqrt{\Delta(N)}$ is in the column space of $N$, which is the one of $M$.

Finally, $\Phi_0$ is not local. Indeed, if $\Phi_0 : M \mapsto MX$ for some $X \in \D^2$, then $X=I_2 X=\Phi_0(I_2)=0$,
so $\Phi_0$ would vanish, which is obviously false.

In contrast with Theorem \ref{theo:specialHermitianbis}, this special case has no extension to greater dimensions, as 
seen in Theorem \ref{theo:mainHermitian}.

\begin{theo}\label{theo:specialHermitianF4}
Let $\F$ be a field with $4$ elements, endowed with its non-identity involution.
Then the range-compatible homomorphisms on $\MatH_2(\F)$
are the sums of the local maps and of the scalar multiples of $\begin{bmatrix}
a & x^\star \\
x & b
\end{bmatrix} \mapsto \begin{bmatrix}
x \\
x^\star
\end{bmatrix}$.
\end{theo}

\subsection{Application to range-compatible linear mappings}

We shall now consider the special situation where we fix a linear subspace $\F$ of the center $\calZ$ of $\D$
for which the involution $(-)^\star$ is $\F$-linear (i.e., $\F$ consists only of central Hermitian elements)
and we assume that $\D$ is finite-dimensional as an $\F$-vector space (which requires $\D$ to be finite-dimensional over its center).
Then $\MatH_n(\D)$ and $\Matah_n(\D)$ are $\F$-linear subspaces of $\Mat_n(\D)$, 
and we may ask what are the range-compatible $\F$-linear mappings on them, observing that every local map
is $\F$-linear. The description is then easier than the more general one from the previous section, and it is precisely this
description that is required for applications to spaces of Hermitian matrices with bounded rank.
Note, for the special case of alternating matrices over a field $\F$, that every range-compatible $\F$-linear mapping
from $\Mata_2(\F)$ to $\F^2$ is local: indeed, every such mapping can be written $\begin{bmatrix}
0 & -x \\
x & 0
\end{bmatrix} \mapsto \begin{bmatrix}
ax \\
bx\end{bmatrix}$ for some $(a,b)\in \F^2$, which is the local mapping $M \in \Mata_2(\F) \mapsto MX$ for $X:=\begin{bmatrix}
b \\
-a
\end{bmatrix}$.

Assume now that $|\F|=2$. Then $\car(\D)=2$, $\D$ is finite (and hence commutative, by Wedderburn's little theorem)
and as $\F$ is the prime subfield of $\D$ the range-compatible homomorphisms are the range-compatible $\F$-linear mappings.
Moreover, in that situation and if $(-)^\star$ is the identity, the root-linear mappings are easy to describe:
indeed, $\D$ is perfect so, if we denote by $\sqrt{-}$ the reciprocal of its Frobenius automorphism $\lambda \mapsto \lambda^2$,
then the root-linear mappings from $\D$ to itself are simply the ones of the form $\lambda \mapsto \sqrt{\lambda}\, a$ for some fixed $a \in \D$.
Hence, we recover the fact that the range-compatible homomorphisms on $\Mats_n(\D)$ are the sums of the local maps with the scalar multiples of 
$M \mapsto \sqrt{\Delta(M)}$.
When $(-)^\star$ is not the identity, every root-linear mapping on $H$ is zero (as $H=A$ in that case, and thanks to Lemma \ref{lemma:rootlinear}), 
so every range-compatible linear mapping
on $\MatH_n(\D)$ is local unless $n=2$ and $|\D|=4$, in which case we have the special case of Theorem \ref{theo:specialHermitianF4}.

Finally, let us consider the situation where $|\F|>2$ and $\car(\D)=2$. In particular this rules out the situation where
$|\D|=4$ and the involution is not the identity.
To start with, we take a mapping $\alpha : H \rightarrow \D$ that is both root-linear and $\F$-linear, and we prove that 
it vanishes.
Indeed, let $a \in H$. By the root-linearity, for all $t \in \F$ we get
$\alpha (t at^\star)=t \alpha(a)$, which reads $t^2 \alpha(a)=t \alpha(a)$. Taking $t \in \F \setminus \{0,1\}$
yields $\alpha(a)=0$, and varying $a$ yields the conclusion.
Let then $\Phi : \MatH_n(\D) \rightarrow \D^n$ be both $\F$-linear and decomposable as the sum of a local mapping and of 
$M \mapsto \Delta(M)^\alpha$ for some root-linear mapping $\alpha : H \rightarrow \D$.
Then, for some $b \in \D$ we have $\forall a \in H, \; \Phi(a E_{1,1})=E_1 (ab+\alpha(a))$ for all $a \in V$, 
and hence $a \mapsto ab+\alpha(a)$ is $\F$-linear. It follows that $\alpha$ is $\F$-linear, and hence $\alpha=0$.
Therefore $\Phi$ is local.

Let us conclude:

\begin{theo}[Full description of all range-compatible $\F$-linear mappings]
Let $\F$ be a field, $\D$ be a finite-dimensional $\F$-algebra that is a division ring, and 
$(-)^\star$ be an $\F$-linear involution of $\D$.
Let $n \geq 2$ be an integer.
Then: 
\begin{enumerate}[(a)]
\item Every range-compatible $\F$-linear mapping on $\Matah_n(\D)$ is local.
\item Unless $(-)^\star$ is the identity and $|\F|=2$, or $|\D|=4$ and $|\F|=n=2$, every range-compatible
$\F$-linear mapping on  $\MatH_n(\D)$ is local.
\item If $|\F|=2$ and $(-)^\star$ is the identity then the range-compatible $\F$-linear mappings on 
$\MatH_n(\D)$ are the sums of the local mappings and of the scalar multiples of $M \mapsto \sqrt{\Delta(M)}$.
\item If $|\D|=4$, $(-)^\star$ is not the identity and $n=2$, then the range-compatible $\F$-linear mappings on 
$\MatH_n(\D)$ are the sums of the local mappings with the scalar multiples of the special mapping 
$\Phi_0 : \begin{bmatrix}
a & x^\star \\
x & b
\end{bmatrix} \mapsto \begin{bmatrix}
x \\
x^\star
\end{bmatrix}$.
\end{enumerate}
\end{theo}

In a subsequent article, this result will be used as a critical tool in the study of 
spaces of Hermitian matrices with bounded rank. We will also use it, in a largely different setting, to decipher large spaces of $\F$-diagonalisable square matrices
with entries in $\D$.

\subsection{Structure of the article}

The proofs of the previous results will be performed as follows: we start from the case $n=2$, and then for larger
values of $n$ we use a localization principle to subgroups of $2$-by-$2$ matrices, and finally we piece the results together to obtain the general case.
The hardest part is the case $n=2$, for which the difficulty is concentrated in the special case $\car(\D)=2$.
Before we tackle this proof, we will start with a general study of root-linear mappings (Section \ref{section:rootlinear}),
and in particular we will prove Theorem \ref{theo:specialHermitianbis} and demonstrate that nontrivial root-linear mappings exist on $H$
whenever $A\neq H$ (still assuming that $\D$ has characteristic $2$).

The $2$-by-$2$ case is done in Section \ref{section:rootlinear}, with a proof that is essentially common to all situations considered in Theorems 
 \ref{theo:mainHermitian}, \ref{theo:mainaltHermitian}, \ref{theo:specialHermitian} 
(but still with $n=2$) and \ref{theo:specialHermitianF4}. The last section is devoted to the generalization to larger values of $n$.

\subsection{Additional notation}

When $n \geq 1$ is fixed and $i,j$ are two elements of $\lcro 1,n\rcro$, we denote by $E_{i,j}$ the matrix of $\Mat_n(\D)$
with exactly one nonzero entry, located at the $(i,j)$-spot and equal to $1_\D$.

\section{On root-linear mappings}\label{section:rootlinear}

\subsection{Basics}

Let us start with a basic result.

\begin{lemma}\label{lemma:rootlinear}
Assume that $\car(\D)=2$. Let $\alpha$ be a root-linear mapping on $H$ or on $A$.
Then $\alpha$ vanishes on $A$.
\end{lemma}

\begin{proof}
The result is obvious if $A=\{0\}$, so we assume that $A$ contains a nonzero element $e$.
Since $\alpha$ is root-linear, we have
$$\forall x \in \D, \; \alpha(xe+ex^\star)=\alpha((x+1)e(x+1)^\star-xex^\star-e)=(x+1)\alpha(e)-x \alpha(e)-\alpha(e)=0.$$
As $e \in \D^\times$, we see that $A=\{(xe)-(xe)^\star \mid x \in \D\}=\{xe+ex^\star \mid x \in \D\}$, and we recover the claimed result.
\end{proof}

Now, let us assume that $\car(\D)=2$, and let us take a root-linear mapping $\alpha$ on $H$.
Then by the previous result $\alpha$ induces a group homomorphism from $H/A$ to $\D$.
Next, we observe that we can endow the quotient abelian group $H/A$ with a structure of left-$\D$-vector space as follows:
the scalar multiplication is simply defined by $\lambda.[h]:=[\lambda h \lambda^\star]$, which is well-defined
because $A$ is invariant under congruence. Checking most of the axioms of left-$\D$-vector spaces is straightforward, with the exception of
the identity $\forall (\lambda,\mu)\in \D^2, \; \forall x \in H/A, \; (\lambda+\mu).x=\lambda.x+\mu.x$, which we will carefully do:
so, let $\lambda,\mu$ belong to $\D$, and let $h \in H$. Then $(\lambda+\mu) h(\lambda+\mu)^\star=\lambda h\lambda^\star+\mu h\mu^\star
+y$ for $y:=\lambda h\mu^\star+\mu h \lambda^\star=(\lambda h \mu^\star)-(\lambda h \mu^\star)^\star \in A$, so we have
$(\lambda+\mu).[h]=\lambda.[h]+\mu.[h]$.
Let us come back to $\alpha$: It induces a group homomorphism $\overline{\alpha} : H/A \rightarrow \D$, 
and the axiom of root-linearity now simply means that $\overline{\alpha}$ is a $\D$-linear mapping between left $\D$-vector spaces.

Conversely, given a nonzero $\D$-linear mapping $\beta : H/A \rightarrow \D$, composing it with the standard projection from $H$ to $H/A$
yields a nonzero root-linear mapping on $H$. By the usual consequences of Zorn's lemma, whenever $A \neq H$ there is a nonzero $\D$-linear mapping from
$H/A$ to $\D$, and consequently there is at least one nonzero root-linear mapping from $H$.
Note however that we have no constructive description of such nonzero mappings, their existence solely results from the Axiom of Choice.
For example, if we take $\D$ as a quaternion division ring over a field $\F$ with characteristic $2$, with the standard quaternion involution,
then $A=\F$ and $H$ is the kernel of the trace, so $A \neq H$, but we have no idea how to construct a root-linear mapping on $H$
without resorting to AC.

\subsection{Root-linear mappings give rise to range-compatible homomorphisms}

Here we shall prove Theorem \ref{theo:specialHermitianbis}.

Assume that $\car(\D)=2$, and let $\alpha : H \rightarrow \D$ be a root-linear mapping.
Let $n \geq 1$. We consider the mapping 
$$\Phi : M \in \MatH_n(\D) \longmapsto \Delta(M)^\alpha.$$
Obviously, it is a group homomorphism. To see that it is range-compatible, we will study first how it behaves with respect to congruence transformations.
So, let $P \in \GL_n(\D)$ and $M=(m_{i,j}) \in \MatH_n(\D)$.
First of all, since $\D$ has characteristic $2$ we can split $M=D+N$ where $D=\mathrm{Diag}(m_{1,1},\dots,m_{n,n})$
and $N \in \Matah_n(\D)$. Then $PMP^\star=PDP^\star+PNP^\star$, with $PNP^\star \in \Matah_n(\D)$. Since the mapping $\alpha$
is root-linear, it vanishes on $A$ and hence $\Phi(PNP^\star)=0$. Therefore,
$$\Phi(PMP^\star)=\Phi(PDP^\star).$$
Next, 
$$\Delta(PDP^\star)=\left(\sum_{j=1}^n p_{i,j}\, m_{j,j}\, p_{i,j}^\star\right)_{1 \leq i \leq n.}$$
By the root-linearity of $\alpha$, this yields
$$\Phi(PMP^\star)=\Phi(PDP^\star)=\left(\sum_{j=1}^n p_{i,j}\, \alpha(m_{j,j})\right)_{1 \leq i \leq n}=P \Phi(M).$$
In particular, $\Phi(M)$ belongs to the column space of $M$ if and only if $\Phi(PMP^\star)$ belongs to the one of $PMP^\star$
(as the latter is the column space of $PM$, by the invertibility of $P^\star$).

Now, take an arbitrary $M \in \MatH_n(\D)$.
Assume first that $\Delta(M)=0$. Then $\Phi(M)=0$ belongs to the column space of $M$.
Assume next that $\Delta(M) \neq 0$. Then it is known that $M=PD P^\star$
for some diagonal matrix $D$ and some $P \in \GL_n(\D)$
(see, e.g, proposition 6.4.2 in \cite{Knus} when $(-)^\star$ is not the identity, and 
theorem 3.0.13 in \cite[Chapter XXXV]{invitquad} otherwise).

By the initial remark, it suffices to observe that $\Phi(D)$ belongs to the column space of $D$, which is obvious
because it has nonzero entries only in positions for which $\Delta(D)$ has nonzero entries. This completes the proof.

Next, assume that $n \geq 2$ and $\Phi=0$. Assume that $\Phi$ equals the local map $M \mapsto MX$ for some $X \in \D^n$, which we write $X=\begin{bmatrix}
x_1 & \cdots & x_n
\end{bmatrix}^T$.
Let $i,j$ be distinct elements of $\lcro 1,n\rcro$.
Then for $M:=E_{i,j}+E_{j,i} \in \MatH_n(\D)$ we have $\Delta(M)=0$ so $MX=0$, which yields $x_i=0$.
Varying $i$ and $j$ yields $X=0$, and then $\Phi=0$.
Finally $0=\Phi(a.I_n)=\begin{bmatrix}
\alpha(a) & \cdots & \alpha(a)
\end{bmatrix}^T$ for all $a \in H$, and we deduce that $\alpha=0$.

Note that the previous result allows us to prove the uniqueness part in Theorem \ref{theo:specialHermitian}. 
Let indeed $\alpha,\beta$ be root-linear maps from $H$ to $\D$, and vectors $X,Y$ in $\D^n$ 
such that $\forall M \in \MatH_n(\F), \; MX+\Delta(M)^\alpha= MY+\Delta(M)^\beta$.
Then clearly $\beta-\alpha$ is root-linear and $\forall M \in \MatH_n(\F),\; M(X-Y)=\Delta(M)^{\beta-\alpha}$,
which leads to $X-Y=0$ and $\beta-\alpha=0$. Hence the claimed uniqueness.

\section{The $2$-by-$2$ case}

Here, we simultaneously prove Theorems \ref{theo:mainaltHermitian}, \ref{theo:mainHermitian} and \ref{theo:specialHermitian}
in the special case $n=2$, as well as Theorem \ref{theo:specialHermitianF4}.

Set $\calM=\MatH_2(\D)$ (respectively, $\calM=\Matah_2(\D)$), 
$\varepsilon=1$ and $V=H$ (respectively, $\varepsilon=-1$ and $V=A$).
In any case $V$ is invariant under congruence. It is also invariant under inversion: This is obvious when $V=H$, but maybe less so if $V=A$.
In the latter case, letting $y \in A \setminus \{0\}$, we observe that $y=-y^\star$ to get $y^{-1}=y^{-1} (-y) (y^{-1})^\star \in V$.

Throughout, we assume that $V \neq \{0\}$.

\subsection{A lemma}

The proof involves the following lemma, which we will prove immediately:

\begin{lemma}\label{lemma:basic}
The mappings $\psi : \D \rightarrow \D$ such that 
$x a \psi(y)^\star=y a \psi(x)^\star$ for all $(x,y,a)\in \D^2 \times V$ are:
\begin{itemize}
\item The zero mapping only, if $\D$ is noncommutative.
\item The mappings of the form $x \mapsto \alpha x^\star$, with $\alpha \in \D$, if $\D$ is commutative.
\end{itemize}
\end{lemma}

\begin{proof}
To start with, it is obvious that $x \mapsto \alpha x^\star$ has the required property if $\D$ is commutative, for all 
$\alpha \in \D$.

Now, we assume that we have a nonzero mapping $\psi : \D \rightarrow \D$ such that 
$x a \psi(y)^\star=y a \psi(x)^\star$ for all $(x,y,a)\in \D^2 \times V$.
We shall prove that $\D$ is commutative and that $\psi : x \mapsto \alpha x^\star$ for some $\alpha \in \D$.
Note that no generality is lost (for these aims) in replacing $\psi$ with $\beta \psi$ for an arbitrary $\beta \in \D^\times$.

Choose now $a \in V \setminus \{0\}$.
Obviously $\psi(0)=0$ (take $x=1$ and $y=0$).
Choose then $y \in \D^\times$ such that $\psi(y) \neq 0$. We recover $a \psi(y)^\star=y a \psi(1)^\star$, which yields $\psi(1) \neq 0$.
As we can replace $\psi$ with $\psi(1)^{-1} \psi$, we lose no generality in assuming that $\psi(1)=1$.
Then we gather from the previous identity that 
\begin{equation}\label{eq:basiclemma}
\forall y \in \D, \; \psi(y)^\star=a^{-1} ya.
\end{equation}
Hence, for all $x,y$ in $\D$ we now have
$$xya=xa \psi(y)^\star=ya \psi(x)^\star=yx a,$$
and since $a \neq 0$ this yields that $\D$ is commutative. And then by \eqref{eq:basiclemma} we find $\psi(y)=y^\star$ for all $y \in \D$, so we have the second part of the stated
outcome.
\end{proof}

\subsection{Main proof}

Now, let $\Phi$ be a range-compatible homomorphism on $\calM$. Recall that we assume $V \neq \{0\}$.

Using the range-compatibility assumption on diagonal matrices, we already find group homomorphisms $\lambda$ and $\mu$ from $V$ to $\D$
such that 
$$\forall (a,b) \in V^2,\; \Phi \begin{bmatrix}
a & 0 \\
0 & 0 
\end{bmatrix}=\begin{bmatrix}
\lambda(a) \\
0
\end{bmatrix} \quad \text{and} \quad 
\Phi \begin{bmatrix}
0 & 0 \\
0 & b 
\end{bmatrix}=\begin{bmatrix}
0 \\
\mu(b)
\end{bmatrix},$$
and we also have group homomorphisms $f : \D \rightarrow \D$ and $g : \D \rightarrow \D$ such that 
$$\forall x \in \D, \; \Phi \begin{bmatrix}
0 & \varepsilon x^\star \\
x & 0
\end{bmatrix}=\begin{bmatrix}
\varepsilon g(x)^\star \\
f(x)
\end{bmatrix}.$$
Throughout, we observe that we can subtract an arbitrary local map from $\Phi$ without affecting the end result nor the assumptions, 
which has the effect of replacing $f$ with $f-\id_\D \alpha$ and $g-\beta \,\id_\D$ where the scalars $\alpha$ and $\beta$ can be chosen at will in $\D$
(of course $\lambda$ and $\mu$ are altered in this process, but this will not be an issue).
Hence, we lose no generality in assuming that $f(1)=0=g(1)$, and we will only change this assumption at a very specific point of the study.

Next, we have to exploit the fact that $\Phi$ maps every rank $1$ matrix to a vector of its range.
Take $a \in V \setminus \{0\}$, and let $x \in \D$.
Then $M=\begin{bmatrix}
a^{-1} & \varepsilon x^\star \\
x & \varepsilon x a x^\star
\end{bmatrix}$ has its range spanned by the vector
$\begin{bmatrix}
1 \\
xa
\end{bmatrix}$, so $\Phi(M)=\begin{bmatrix}
\lambda(a^{-1})+\varepsilon g(x)^\star \\
f(x)+\mu(\varepsilon x a x^\star)
\end{bmatrix}$ is a right multiple of this vector. Hence 
\begin{equation}\label{eq:basicid}
\forall x \in \D, \; \forall a \in V \setminus \{0\}, \quad f(x)+\mu(\varepsilon x a x^\star)=xa(\lambda(a^{-1})+\varepsilon g(x)^\star).
\end{equation}
We shall now take advantage of this identity to decipher the various mappings under consideration.
We shall split the discussion into several cases, by increasing order of difficulty.

\vskip 3mm
\noindent \textbf{Case 1:} $\car(\D) \neq 2$. \\
Denote by $\F_0$ the prime subfield of $\D$, so that $|\F_0|>2$. We view $V$ and $\D$ as $\F_0$-vector spaces, and we see 
that $f,g,\lambda,\mu$ are all $\F_0$-linear. Let $a \in V \setminus \{0\}$.
Fix $x \in \D$. Applying identity \eqref{eq:basicid} to $t x$ for an arbitrary $t \in \F_0$, 
we find $t f(x)+t^2 \varepsilon \mu(x a x^\star)=txa \lambda(a^{-1})+t^2\varepsilon x a g(x)^\star$.
This is a polynomial identity in the variable $t$ in the $\F_0$-vector space $\D$.
Since $|\F_0|>2$ we can extract the identity $f(x)=xa \lambda(a^{-1})$. 
Applying this to $x=1$, we obtain $\lambda(a^{-1})=0$ for all $a \in V\setminus \{0\}$, and hence $\lambda=0$.
And then $f=0$ by picking a nonzero $a$ once more, and an arbitrary $x$.
Symmetrically, we obtain $\mu=0$ and $g=0$. Then $\Phi$ is the zero mapping, and we conclude that it is local
(it is the right multiplication by the zero vector).

\vskip 3mm
\noindent \textbf{Case 2:} $\car(\D) = 2$. \\
We come back to identity \eqref{eq:basicid} to obtain, for all $a \in V \setminus \{0\}$,
$$\forall x \in \D, \; f(x)+xa\lambda(a^{-1})=\mu(x a x^\star)-xa g(x)^\star.$$
By polarizing, we deduce that 
$$\forall a \in V \setminus \{0\}, \; \forall (x,y)\in \D^2, \; \mu(x a y^\star+y a x^\star)=x a g(y)^\star+y a g(x)^\star,$$
which obviously holds also for $a=0$.
Now, fix $z \in \D^\times$.
Let $a \in V$. Then for all $(x,y)\in \D^2$ we find
\begin{align*}
(xz) a g(yz)^\star+(yz) a g(xz)^\star
& =\mu((xz) a (yz)^\star+(yz) a (xz)^\star) \\
& =\mu(x(zaz^\star)y^\star+y (zaz^\star) x^\star) \\
& =x (z a z^\star) g(y)^\star+y (z a z^\star) g(x)^\star.
\end{align*}
Applying this to $a':=z^{-1} a (z^\star)^{-1}$ leads to 
$$x a (g(yz) z^{-1})^\star+y a (g(xz) z^{-1})^\star=x a g(y)^\star+y a g(x)^\star.$$
Hence, for $\varphi_z : x \mapsto g(xz) z^{-1}-g(x)$, we have proved 
that $x a \varphi_z(y)^\star=y a \varphi_z(x)^\star$ for all $x,y$ in $\D$ and all $a \in V$,
and this is of course related to Lemma \ref{lemma:basic}.

We shall now split the discussion into two subscases, whether $\D$ is commutative or not.

\vskip 3mm
\noindent \textbf{Subcase 2.1:} $\D$ is noncommutative. \\
By applying Lemma \ref{lemma:basic} to $\varphi_z$, we deduce that $g(xz) =g(x)z$ for all $x \in \D$ and all $z \in \D^\times$ (and of course also for $z=0$). Applying this to $x=1$ and varying $z$ leads to $g=0$. Symmetrically, we obtain $f=0$.
At this point, identity \eqref{eq:basicid} can be rewritten
\begin{equation}\label{eq:basicidsubcase2.1}
\forall a \in V \setminus \{0\}, \; \forall x \in \D, \quad 
x a \lambda(a^{-1})=\mu(xax^\star).
\end{equation}
Taking $x=a^{-1}$, this yields $\forall a \in V \setminus \{0\}, \; \lambda(a^{-1})=\mu(a^{-1})$, and hence $\lambda=\mu$.
Next, we obtain $\forall (x,a) \in \D \times (V \setminus \{0\}), \; \lambda (xax^\star)=x a \lambda(a^{-1})$, the special case $x=1$ yields
$\forall a \in V \setminus \{0\}, \; \lambda(a)=a \lambda(a^{-1})$ and we conclude from \eqref{eq:basicidsubcase2.1} that $\lambda$ is root-linear.
If in addition $V=A$ then $\lambda=0$ by Lemma \ref{lemma:rootlinear}, to the effect that $\Phi=0$.
Otherwise we can still write $\Phi : M \mapsto \Delta(M)^\lambda$.

\vskip 3mm
\noindent \textbf{Subcase 2.2:} $\D$ is commutative. \\
We further split the discussion into two subcases, whether $(-)^\star$ is the identity or not.

\vskip 3mm
\noindent \textbf{Subsubcase 2.2.1:} $(-)^\star$ is the identity. \\
This case has already been dealt with in \cite{dSPRC1}, but since its treatment is very short we will repeat the argument.
We go right back to identity \eqref{eq:basicid}, apply it to $a=1$ and rewrite the result as
$$\forall x \in \D, \; f(x)+\mu(x^2)+x\lambda(1)=xg(x).$$
The left-hand side is an additive function of $x$.
By polarizing, we obtain $xg(y)+y g(x)=0$ for all $x,y$ in $\D$. Taking $x=1$, we deduce that $g=0$.
Symmetrically, we find $f=0$. The conclusion is then derived as in Subcase 2.1.

\vskip 3mm
\noindent \textbf{Subsubcase 2.2.2:} $(-)^\star$ is not the identity. \\
Here $V=H=A$.
Let $z \in \D^\times$. Again, we apply Lemma \ref{lemma:basic} to $x \mapsto z^{-1} g(xz)-g(x)$, and we recover an element
$\alpha(z) \in \D$ such that $\forall y \in \D, \; g(yz)=z g(y)+\alpha (z) y^\star$.
This of course holds also for $z=0$ by taking $\alpha(0):=0$.
With $y=1$ and varying $z$, we find $\alpha=g$. Hence $\forall (y,z) \in \D^2, \; g(yz)=z g(y)+g(z) y^\star$.
Since $\D$ is commutative, for all $(y,z) \in \D^2$ we obtain
$zg(y)+g(z) y^\star=y g(z)+g(y)z^\star$ and hence $g(y)(z-z^\star)=g(z) (y-y^\star)$.
By choosing $z$ such that $z-z^\star \neq 0$ (this is possible since $(-)^\star$ is not the identity),
we deduce that $g : x \mapsto \beta (x^\star-x)$ for some $\beta \in \D$.
Symmetrically, $f : x \mapsto \alpha (x^\star-x)$ for some $\alpha \in \D$.
At this point, we will change the fundamental assumption we started from: we add
$x \mapsto \alpha x$ to $f$ and $x \mapsto \beta x$ to $g$. Beware then that we no longer have $f(1)=0=g(1)$, 
but we have the simplified expressions $f : x \mapsto \alpha x^\star$ and $g : x \mapsto \beta x^\star$.

Now, we shall prove that $\lambda=\mu=0$ in this situation. We come right back to identity \eqref{eq:basicid} and obtain
$$\forall x \in \D,\; \forall a \in H \setminus \{0\}, \; 
\mu(xx^\star a)=xa\lambda(a^{-1})+\beta^\star a x^2+\alpha x^\star.$$
Let us fix $a \in  H \setminus \{0\}$. Here the right-hand side is an additive function of $x$,
so by polarizing we find $\mu((xy^\star+y x^\star)a)=0$ for all $x,y$ in $H$.
The special case $y=1$ and $a=1$ yields $\mu(a')=0$ for all $a' \in A=H=V$. Hence $\mu=0$, and
symmetrically we obtain $\lambda=0$.
We are almost ready to conclude. Now, with $a=1$ identity \eqref{eq:basicid} is simplified as
$$\forall x \in \D, \; \beta^\star x^2=\alpha x^\star.$$
With $x=1$ we deduce that $\alpha=\beta^\star$.

If $\alpha=0$ we deduce from having $\beta=\alpha^\star=0$ that $\Phi=0$. Assume finally that $\alpha \neq 0$.
Hence the previous identity yields $\forall x \in \D, \; x^2=x^\star$, and then $\forall x \in \D, \; x^4=(x^2)^2=(x^\star)^\star =x$.
Since $\D$ is a field, this yields $|\D| \leq 4$. Of course we must then have $|\D|=4$ because $(-)^\star$ is not the identity.
Finally, we have $f : x \mapsto \alpha x^\star$, $g : x \mapsto \alpha^\star x^\star$, and $\lambda=\mu=0$, 
and we conclude that $\Phi=\Phi_0\, \alpha$. 

This completes the proof of all Theorems \ref{theo:mainHermitian}, \ref{theo:mainaltHermitian}, and \ref{theo:specialHermitian} 
in the special case $n=2$, as well as the one of Theorem \ref{theo:specialHermitianF4}.

\section{Matrices with more than two rows}

We will now finish the proof of the main theorems by using the case $n=2$ to derive the one of larger values of $n$.

We have to consider three distinct situations. We will start with the easier one, in which 
the characteristic of $\D$ is not $2$. Then, we will consider the characteristic $2$ case, excluding the special 
case where $|\D|=4$ and $(-)^\star$ is not the identity. We will finish with this special case, which is slightly more difficult.

In any case, we take an integer $n \geq 3$, we set
$\calM=\MatH_n(\D)$ and $\calN=\MatH_2(\D)$ (respectively $\calM=\Matah_n(\D)$ and $\calN=\Matah_2(\D)$ if $(-)^\star$ is not the identity).
We set $\varepsilon:=1$ (respectively $\varepsilon:=-1$) and $V:=H$ (respectively $V:=A$),
and we let $\Phi$ be a range-compatible homomorphism on $\calM$.

For $i \in\lcro 1,n\rcro$ and $M \in \calM$, we denote by $\Phi(M)_i$ the $i$-th entry of the column vector $\Phi(M)$, and we recall that $(E_1,\dots,E_n)$
denotes the standard basis of the right $\D$-vector space $\D^n$.

Let $i$ and $j$ be distinct elements of $\lcro 1,n\rcro$.
For each $(a,b,x)\in V^2 \times \D$, the vector
$\Phi(E_{i,i}\,a+E_{j,j}\,b+ E_{j,i}\,x+E_{i,j}\,\varepsilon x^\star)$ must belong to the range of 
$M:=E_{i,i}\,a+E_{j,j}\,b+ E_{j,i}\,x+E_{i,j}\,\varepsilon x^\star$, and hence it equals
$E_i\,\Phi(M)_i+E_j\,\Phi(M)_j $.
Better still, the mapping
$$\Phi_{i,j} : \begin{bmatrix}
a & \varepsilon x^\star \\
x & b
\end{bmatrix} \in \calN \mapsto \begin{bmatrix}
\Phi(E_{i,i}\,a+E_{j,j}\,b+ E_{j,i}\,x+E_{i,j}\,\varepsilon x^\star)_i \\
\Phi(E_{i,i}\,a+E_{j,j}\,b+ E_{j,i}\,x+E_{i,j}\,\varepsilon x^\star)_j
\end{bmatrix}$$
is a range-compatible homomorphism, which will allow us to use the results of the $2$-by-$2$ case.

\subsection{Easy cases}\label{section:easycases}

Here, we assume that every range-compatible homomorphism on $\MatH_2(\D)$ (respectively, on $\Matah_2(\D)$) is local, and 
we shall prove that this extends to $\Mat_n(\D)$. This will complete the proofs of Theorems \ref{theo:mainaltHermitian} and \ref{theo:mainHermitian}.

We choose $a_0 \in V \setminus \{0\}$, which is possible in any case.
Then for all $i \in \lcro 1,n\rcro$ we have a scalar $x_i \in \D$ such that $\Phi(E_{i,i}\, a_0)=E_i\,a_0 x_i$.
We can then set $X:=\begin{bmatrix}
x_1 & \cdots & x_n
\end{bmatrix}^T$ and subtract the local mapping $M \mapsto MX$ from $\Phi$ to reduce the situation to the one where
$\Phi$ vanishes at $E_{i,i}\, a_0$ for all $i \in \lcro 1,n\rcro$. In that reduced situation we shall prove that $\Phi$ is zero, and in particular
it will be a local mapping.

Let $i,j$ be distinct elements of $\lcro 1,n\rcro$. We apply the $2$-by-$2$ case to 
$\Phi_{i,j}$: It follows that $\Phi_{i,j}$ is local, which yields a pair $(\alpha,\beta)\in \D^2$ such that 
$\Phi_{i,j}(N)=N \begin{bmatrix}
\alpha \\
\beta
\end{bmatrix}$ for all $N \in \calN$. In particular $\Phi(E_{i,i}\, a_0)=a_0 \alpha$ and $\Phi(E_{j,j}\, a_0)=a_0 \beta$,
and we deduce that $\alpha=0=\beta$. In turn, this yields $\Phi(E_{j,i} x+E_{i,j}\,\varepsilon x^\star)=0$ for all $x \in \D$,
and $\Phi(E_{i,i}\, a)=0$ for all $a \in V$.
Varying $i$ and $j$ and using the fact that $\Phi$ is additive, we conclude that $\Phi=0$.

\subsection{The special case of fields with characteristic $2$, for Hermitian matrices}

Now, in order to prove Theorem \ref{theo:specialHermitian} for the integer $n$,  we consider the case where $\car(\D)=2$, and we take
$\calM=\MatH_n(\D)$ and $\calN=\MatH_2(\D)$. We rule out the case where $|\D|=4$ and $(-)^\star$ is not the identity, 
which allows us to apply the case $n=2$ in Theorem \ref{theo:specialHermitian}.

Let $i$ and $j$ be distinct elements of $\lcro 1,n\rcro$. Again, we apply the $2$-by-$2$ case to 
$\Phi_{i,j}$. This yields a uniquely-defined vector $X=\begin{bmatrix}
a_{i,j} \\
b_{i,j}
\end{bmatrix} \in \D^2$ and a uniquely-defined root-linear mapping $\alpha_{i,j} : H \rightarrow \D$ such that 
$$\forall N \in \MatH_2(\D), \; \Phi_{i,j}(N)=NX+\Delta(N)^{\alpha_{i,j}}.$$
This yields $\forall x\in \D, \; \Phi(E_{j,i}\, x+ E_{i,j}\, x^\star)=E_j\, x a_{i,j} +E_i\, x^\star b_{i,j}$.
We shall prove that $a_{i,j}$ depends only on $i$.
Let indeed $i,j,j'$ be distinct elements of $\lcro 1,n\rcro$.
We have $\Phi(E_{j,i}+E_{i,j})=E_j\, a_{i,j}+  E_i \, b_{i,j}$ 
and $\Phi(E_{j',i}+E_{i,j'})= E_{j'}\, a_{i,j'} + E_i\, b_{i,j'}$, 
and hence 
$$\Phi(E_{j,i}+E_{j',i}+E_{i,j}+E_{i,j'})= E_j\, a_{i,j}+ E_{j'}\, a_{i,j'}+E_i\, (b_{i,j}+b_{i,j'}).$$ 
Noting that the range of 
$E_{j,i}+E_{j',i}+E_{i,j}+E_{i,j'}$ is the linear span of
$E_i$ and $E_j+E_{j'}$, we deduce from the range-compatibility assumption that $a_{i,j}=a_{i,j'}$, as claimed.
Finally, it is clear that $b_{i,j}=a_{j,i}$ for all distinct $i,j$ in $\lcro 1,n\rcro$.
Hence, we recover a list 
$X=\begin{bmatrix}
a_1 & \cdots & a_n
\end{bmatrix}^T$ of elements of $\D$ such that 
$\Phi(E_{j,i}\,x+ E_{i,j}\, x^\star )=E_i\, x a_i+E_j\, x^\star a_j$ for all distinct $i,j$ in $\lcro 1,n\rcro$ and all $x \in \D$.

Hence, by subtracting from $\Phi$ the local mapping $M \mapsto MX$, we reduce the situation to the one where 
$\Phi(E_{j,i}\, x+ E_{i,j}\, x^\star)=0$ for all distinct $i,j$ in $\lcro 1,n\rcro$.

Finally, we come back to the analysis of the $\Phi_{i,j}$ mappings. Let $i,j$ be distinct elements of $\lcro 1,n\rcro$.
Now we have $a_{i,j}=b_{i,j}=0$, because $\Phi(E_{j,i}+E_{i,j})=0$.
It follows that we have a root-linear mapping $\alpha_{i,j} : H \rightarrow \D$ such that 
$\Phi_{i,j}(N)=\Delta(N)^{\alpha_{i,j}}$ for all $N \in \MatH_2(\D)$. In particular 
$\Phi(E_{i,i}\, a)=E_i \, \alpha_{i,j}(a)$ and $\Phi(E_{j,j}\,a)=E_j\, \alpha_{i,j}(a)$ for all $a \in H$.
In particular we have root-linear mappings $\alpha_1,\dots,\alpha_n$ on $H$ such that 
$\Phi(E_{i,i}\, a)=E_i\, \alpha_i(a) $ for all $i \in \lcro 1,n\rcro$ and all $a \in H$, and clearly for all distinct $i,j$ in $\lcro 1,n\rcro$
we have $\alpha_i=\alpha_{i,j}=\alpha_j$. Hence for $\alpha:=\alpha_1$ we have $\alpha_i=\alpha$ for all $i \in \lcro 1,n\rcro$.
We conclude, since $\Phi$ is additive, that $\Phi : M \mapsto \Delta(M)^\alpha$, which completes the proof.

\subsection{The special case of $\F_4$}

Here we assume that $|\D|=4$ and $(-)^\star$ is not the identity of $\D$.
Hence $\D$ is a finite field with cardinality $4$, which we write $\F_4$,
and $H$ is the prime subfield of $D$, which we denote by $\F_2$.
Note that $H$ has only one nonzero element.

We must prove that every range-compatible homomorphism on $\MatH_n(\D)$
is local. Since this fails for $n=2$ we will have to settle the case $n=3$ first, and then the case of larger integers will be deduced from the 
case $n=3$ by using a localization technique that is similar to what was used earlier.

So, we consider first the case $n=3$. We wish to prove that $\Phi$ is local.
We have scalars $a_1,a_2,a_3$ in $\F_4$ such that $\Phi(E_{i,i})=E_i\, a_i$ for all $i \in \lcro 1,3\rcro$.
Subtracting from $\Phi$ the local mapping $M \mapsto MX$ for $X:=\begin{bmatrix}
a_1 & a_2 & a_3
\end{bmatrix}^T$, we lose no generality in assuming that $\Phi$ vanishes at $E_{1,1},E_{2,2}$ and $E_{3,3}$.
Then, let $i,j$ be distinct elements of $\lcro 1,3\rcro$, and let us consider the localized mapping
$\Phi_{i,j} : \MatH_2(\F_4) \rightarrow (\F_4)^2$. By Theorem \ref{theo:specialHermitianF4}, there is a scalar $\alpha \in \F_4$ and a vector 
$Y \in (\F_4)^2$ such that $\Phi_{i,j} : N \mapsto NY+\Phi_0(N) \alpha$ (recall the notation $\Phi_0$ from Section \ref{section:mainresults}).
Since $\Phi_0$ vanishes at every diagonal matrix, and since $\Phi_{i,j}$ also does, we obtain $Y=0$.
Hence $\Phi(E_{j,i}\,x+E_{i,j}\, x^\star)=E_i\, x \alpha +E_j\, x^\star \alpha$ for all $x \in \F_4$.
Varying $\{i,j\}$, we obtain three scalars $\alpha,\beta,\gamma$ such that 
$$\Phi : \begin{bmatrix}
? & x^\star & y^\star \\
x & ? & z^\star \\
y & z & ?
\end{bmatrix} \longmapsto \begin{bmatrix}
\alpha x+\beta y \\
\alpha x^\star +\gamma z \\
\beta y^\star+\gamma z^\star
\end{bmatrix}.$$
Taking the special matrix with all entries equal to $1$, which has rank $1$, we deduce that 
$\alpha+\beta=\alpha+\gamma=\beta+\gamma$, which yields $\alpha=\beta=\gamma$.
Finally, we take $x \in \F_4 \setminus \F_2$ and consider the special matrix
$M=X^\star X=\begin{bmatrix}
1 & x & x \\
x^\star & 1 & 1 \\
x^\star & 1 & 1
\end{bmatrix}$, for $X:=\begin{bmatrix}
1 & x & x
\end{bmatrix}$, whose column space is spanned by $X^\star=\begin{bmatrix}
1 \\
? \\
?
\end{bmatrix}$. Yet $\Phi(M)=\alpha \begin{bmatrix}
0 \\
x+1 \\
x+1
\end{bmatrix}$
with $x+1 \neq 0$.
It follows that $\alpha=0$, and we conclude that $\Phi=0$.

Let us complete the proof by returning to the more general case where $n \geq 3$.
Again, we want to prove that $\Phi$ is local.
As in Section \ref{section:easycases} we subtract a well-chosen local mapping to reduce the situation to the one where $\Phi(E_{i,i})=0$
for all $i \in \lcro 1,n\rcro$. Let then 
$i_1,i_2,i_3$ be distinct integers in $\lcro 1,n\rcro$. For all $N=(n_{k,l}) \in \MatH_3(\F_4)$, consider the Hermitian matrix
$\widetilde{N}^{(i_1,i_2,i_3)}:=\sum_{1\leq k,l \leq 3} E_{i_k,i_l}\,n_{k,l} \in \MatH_n(\F_4)$,
and the mapping 
$$\Phi_{i_1,i_2,i_3} : N \in \MatH_3(\F_4) \longmapsto \begin{bmatrix}
\Phi(\widetilde{N}^{(i_1,i_2,i_3)})_{i_1} \\
\Phi(\widetilde{N}^{(i_1,i_2,i_3)})_{i_2} \\
\Phi(\widetilde{N}^{(i_1,i_2,i_3)})_{i_3}
\end{bmatrix} \in (\F_4)^3,$$ 
which is clearly a range-compatible homomorphism.
By the previous part of the proof, we find that $\Phi_{i_1,i_2,i_3}$ is local.

Like in Section \ref{section:easycases}, we combine this with the observation that 
$\Phi(E_{i_k,i_k})=0$ for all $k \in \{1,2,3\}$ to deduce that $\Phi_{i_1,i_2,i_3}=0$.
In particular $\Phi(E_{i_1,i_2}\, x+ E_{i_2,i_1}\, x^\star)=0$ for all $x \in \F_4$.
Varying $(i_1,i_2,i_3)$ yields, since $\Phi$ is additive, that $\Phi=0$, and the proof is complete.

\end{document}